\documentclass[12pt,a4paper]{article}

\usepackage{amssymb}
\usepackage{amsmath}
\usepackage{amsthm}

\newtheorem{theorem}{Theorem}
\newtheorem{lemma}{Lemma}

\newcommand{\G}{\ensuremath{\Gamma}} 
\newcommand{\Gr}{G}                  

\begin{document}

\title{There is No McLaughlin Geometry}

\author{Patric R. J. \"Osterg{\aa}rd\\
Department of Communications and Networking\\
Aalto University School of Electrical Engineering\\
P.O.\ Box 13000, 00076 Aalto, Finland\\
{\small \texttt{patric.ostergard@aalto.fi}}
\and 
Leonard H. Soicher\\
School of Mathematical Sciences\\
Queen Mary University of London\\
Mile End Road\\
London E1 4NS, UK\\
{\small \texttt{l.h.soicher@qmul.ac.uk}}
}

\date{}

\maketitle

\begin{abstract}
We determine that there is no partial geometry ${\cal G}$ with parameters
$(s,t,\alpha)=(4,27,2)$.  The existence of such a geometry has been
a challenging open problem of interest to researchers for almost $40$
years.  The particular interest in ${\cal G}$ is due to the fact that
it would have the exceptional McLaughlin graph as its point graph.
Our proof makes extensive use of symmetry and high-performance
distributed computing, and details of our techniques and checks are
provided. One outcome of our work is
to show that a pseudogeometric strongly regular graph achieving equality
in the Krein bound need not be the
point graph of any partial geometry.
\end{abstract}

 

\newpage

\section{Introduction}

For positive integers $s,t,\alpha$, a \emph{partial geometry} with
\emph{parameters} $(s,t,\alpha)$, or a pg$(s,t,\alpha)$, is
a point-line incidence structure such that
\begin{enumerate}
\item 
there are exactly $s+1$ points on each line;
\item
there are exactly $t+1$ lines through each point;
\item 
every pair of distinct points are on at most one line
(and so every pair of distinct lines meet in at most one point); and
\item 
for every line $L$
and every point $p$ not on $L$, there are exactly $\alpha$ lines through
$p$ meeting $L$. 
\end{enumerate}
Omitting the last condition, we obtain the definition of a 
\emph{partial linear space} with \emph{parameters} $(s,t)$, 
or a pls$(s,t)$. 

Partial geometries, introduced by Bose \cite{B}, 
generalise certain important classes of finite geometries:
generalised quadrangles, 2-$(v,s+1,1)$ designs, nets,
and the duals of these geometries. 
See \cite{DeCM} and \cite{T} for useful surveys on 
partial geometries. 

The \emph{point graph} of a partial linear space ${\cal P}$ 
has the points of ${\cal P}$ as vertices
and has $\{p,q\}$ as an edge if and only if 
the points $p$ and $q$ are distinct and on a common line of ${\cal P}$.
It is well known (and first proved in \cite{B}) that a 
pg$(s,t,\alpha)$ has a strongly regular point graph, with parameters
\begin{equation}
\label{pseugeom}
\left( (s+1)(st+\alpha)/\alpha,\,s(t+1),\,s-1+t(\alpha-1),\,(t+1)\alpha\right).
\end{equation}
A strongly regular graph
with parameters of the form $(\ref{pseugeom})$, 
with $s,t,\alpha\in\mathbb{Z}$,  $1\le s,t$ and $1\le \alpha\le\mathrm{min}\{s+1,t+1\}$, 
is called a \emph{pseudogeometric}
$(s,t,\alpha)$-\emph{graph}, and such a graph is
called \emph{geometric} if it is the point graph of at least one
pg$(s,t,\alpha)$. It is known
that if a pseudogeometric $(s,t,\alpha)$-graph $\Delta$ is the point graph of 
a pls$(s,t)$, then this pls$(s,t)$ must be a pg$(s,t,\alpha)$ and 
$\Delta$ is geometric. This follows from the fact that 
if a pseudogeometric $(s,t,\alpha)$-graph has an $(s+1)$-clique $C$, 
then every vertex not in $C$ is adjacent to exactly $\alpha$ 
vertices in $C$  (see \cite[Lemma~1]{H} and its proof).

It has long been an open problem of considerable interest whether 
there exists a partial geometry ${\cal G}$ with parameters 
$(s,t,\alpha)=(4,27,2)$
(see \cite{CGS,DeCM,v,M97,S}). Such a geometry would
have $275$ points and $1540$ lines, and the point graph $\G$ of 
${\cal G}$ would be a strongly regular graph with parameters
\[(v,k,\lambda,\mu)=(275,112,30,56).\]
Goethals and Seidel \cite{GS} showed that such a strongly regular graph 
is unique (up to isomorphism). Thus $\G$ would be the 
well-studied and fascinating 
McLaughlin graph, defined by J.~McLaughlin \cite{McL}
for the construction of his sporadic simple group, now
called the McLaughlin group $McL$. 
The full automorphism group of $\G$ is $\Gr\cong McL{:}2$,
which acts transitively with permutation rank $3$ on the set of vertices
\cite{McL}. 

As in \cite{S}, we call a pg$(4,27,2)$ a 
\emph{McLaughlin geometry}. 
The main purpose of this paper is finally to settle 
the existence of a McLaughlin geometry, unfortunately in the negative.
We do this by proving that the 
McLaughlin graph is not the point graph of any pls$(4,27)$.  

The McLaughlin graph $\G$ is an example of a pseudogeometric 
$(s,t,\alpha)$-graph achieving equality in the \emph{Krein bound} 
\[
(s + 1- 2\alpha)t \le  (s - 1)(s + 1 - \alpha)^2
\]  
(see \cite[Theorem~7.6]{CGS}).
The problem of the existence of a McLaughlin geometry appears first
to have been posed in \cite{CGS}, and was in connection with   
\cite[Question~7.10]{CGS} asking whether every pseudogeometric
$(s,t,\alpha)$-graph achieving equality in the Krein bound is
geometric. It was proved in \cite[Theorem~7.9]{CGS} that this
is indeed the case when $\alpha=1$, so the McLaughlin graph 
was a particularly important test case. 
See also the Remark and Open Question 
in \cite[p.~442]{DeCM}.
The nonexistence of a McLaughlin geometry now shows that a pseudogeometric 
$(s,t,\alpha)$-graph achieving equality in the Krein bound need not
be geometric when $\alpha>1$. 


For brevity, throughout this paper we omit the term \emph{putative}
when writing about putative McLaughlin geometries and their
substructures. 
For our proof, we fix a copy $\G$ 
of the McLaughlin graph. In the Appendix, we give explicit
permutation generators for the automorphism group
$\Gr$ of $\G$, together with the construction of
$\G$ from $\Gr$.
Although we focus on one specific pseudogeometric graph,
the main ideas and techniques discussed
are applicable more generally. 

\section{Candidate lines of a McLaughlin geometry}

Given our copy $\G$ of the McLaughlin graph, the first step
is to determine the candidates for lines of a McLaughlin 
geometry. Each candidate line is a 
$5$-clique in $\G$. There are exactly 15400 5-cliques in $\G$,
which can be readily obtained using Cliquer \cite{NO} or \textsf{GRAPE} \cite{GRAPE}
(see, for example, the Appendix).
We denote the set of these 15400 cliques (or \emph{candidate lines}, or simply \emph{lines})  
by ${\cal C}$.

We now give some basic properties of these cliques, which
can also be easily checked computationally.

\begin{lemma}
\label{lem:cl}
\begin{enumerate}
\item Each vertex in\/ $\G$ is in exactly \/ $280$ of the cliques in\/ ${\cal C}$.
\item Each edge in\/ $\G$ is contained in exactly \/ $10$ of the cliques in\/ ${\cal C}$.
\item Each\/ $k$-clique in\/ $\G$ with\/ $3 \leq k \leq 5$ is contained in exactly
one of the cliques in\/ ${\cal C}$.
\item There are no\/ $k$-cliques in\/ $\G$ of size\/ $k>5$.
\end{enumerate}
\end{lemma}

\begin{proof}
The group $\Gr$ acts transitively on both the vertices and
the edges of $\G$. 
Consequently, each of the $275\cdot 112/2 = 15400$ edges
of $\G$ occur in $15400\cdot 10/15400 = 10$ of the cliques
in ${\cal C}$, and each of the 275 vertices of $\G$ occur in
$15400\cdot 5/275 = 280$ of the cliques in ${\cal C}$.

Since $\G$ is a pseudogeometric $(4,27,2)$-graph, 
each vertex not in a given 5-clique is adjacent to exactly two vertices
in that clique.
It is shown in \cite{v} that each 3-clique is contained in some 5-clique
(or see the Appendix).
This 5-clique $C$ is unique, since otherwise
there would be a vertex outside $C$ that is adjacent to at least
three vertices of $C$, which is not possible.
This implies that also each 4-clique is contained in a unique 5-clique and that
there can be no clique in $\G$ of size greater than $5$. 
\end{proof}

The existence problem for a pg$(4,27,2)$ boils down to 
the problem of finding a subset of size 1540 of the set ${\cal C}$
of 5-cliques such that each edge of the McLaughlin graph
is in exactly one of the chosen 5-cliques. For this, the framework of
exact cover \cite{KP,K} could be considered, and indeed this approach 
is used in \cite{M97}.  However, many further ideas are needed to solve
the existence problem of a McLaughlin geometry, 
even with the large available computing resources.
Indeed, it is clear from earlier studies \cite{v,M97,S} that constructing a McLaughlin 
geometry, or proving nonexistence, is a big challenge. As with other
similar challenging problems \cite{KO}, care has to be taken to
find an efficient approach that utilizes the available symmetries.

\section{The general approach}

Our general strategy follows that described by 
Reichard \cite{R}. This strategy was independently devised 
by the second author, and implemented
in his \textsf{GRAPE} package \cite{GRAPE} function \texttt{Partial\-Linear\-Spaces},
which can classify partial linear spaces, including partial geometries, 
with given parameters and point graph, up to isomorphism. 
In the present work, however, we needed to employ new ideas and enhancements, 
implemented on a powerful 
256-core computer cluster, to be able finally to solve the existence 
problem of a McLaughlin geometry. 

We first define the concept of a bundle. If $p$ is a point
in a McLaughlin geometry, then the 28 lines through $p$ 
meet pairwise only in $p$, for otherwise, there would
be two distinct points on more than one line. Where $p$ is
a vertex in the  McLaughlin graph
$\G$ (we also call $p$ a \emph{point}), then a \emph{bundle} through $p$ is a set 
of 28 5-cliques whose pairwise intersection is $\{p\}$. 

Our basic approach for an exhaustive search for a McLaughlin
geometry is to consider the points $1,2,\ldots,275$ of $\G$, and
given chosen bundles through \mbox{$i-1$} chosen points, $p_1,\ldots,p_{i-1}$,
to choose a new point $p_i$ (which may depend on the points and bundles chosen so far), 
then consider all the possible bundles 
through $p_i$.  Some bundles through $p_i$ may be incompatible with previous choices, 
or may not be necessary to consider by symmetry or since they were eliminated 
from consideration in a previous case, 
but for each bundle through $p_i$ that it is necessary to consider, we
continue the search using a new point $p_{i+1}$  
not in $\{p_1,\ldots,p_i\}$. The full search starts with $i=1$. 

\subsection{Classifying bundles}

Since the group $\Gr$ acts transitively on the points, the
first point $p_1$ can be chosen arbitrarily, so we choose $p_1=1$,
and start by classifying the bundles through this point, up to the action
of the stabilizer $U$ in $\Gr$ of $p_1$. 
This stabilizer has shape $U_4(3).2$ \cite{McL} and order 
$|\Gr|/275 = 2^8\cdot 3^6 \cdot 5 \cdot 7 = 6531840$.

There are exactly 17729280 bundles through $p_1$, forming just 36 $U$-orbits. 
This was already discovered
by Pech and Reichard \cite[Sect.\ 8.2]{PR}, but no further
details about the structures were given in that study. 

We determined the 36 $U$-orbits of the 
17729280 distinct bundles through $p_1$ using exact 
cover \cite{KP}: each 5-clique that contains
$p_1$ covers the four other
points on the candidate line, and we need to 
determine the sets (of size 28) of such 5-cliques that cover 
each the 112 vertices adjacent to $p_1$ exactly once.
The Orbit--Stabilizer theorem was applied for validating the results; cf.\
\cite[Sect.\ 10.2]{KO}. More precisely,
it was checked that the number of bundles in each orbit 
determined during the search is $6531840/|A|$, where $|A|$ is the
order of the stabilizer in $U$ of an orbit representative.  

We remark that the bundles through $p_1$ correspond
to the ``spreads" of the pg$(3,9,1)$ (or generalised quadrangle GQ$(3,9)$), whose 
point graph is the induced subgraph $\Delta$ on the neighbors in $\G$ of $p_1$, and that 
these spreads form just $26$ orbits under the action of the
automorphism group of $\Delta$, of shape $U_4(3).D_8$  (see \cite{B4} 
 and \cite[Sect.\ 8.2]{PR}). 

We use the 36 \emph{lex-min}
$U$-orbit representatives of the bundles in subsequent computations. This is
an important detail, since, for example, references to specific points
are not valid for arbitrary representatives. The natural order $1,2,\ldots,275$ of the points
naturally induces a lexicographic order on the candidate lines (which are sets of points),
and so also on the sets of candidate lines. 
The lex-min $U$-orbit representatives of the bundles through
$p_1$ were computed independently by both authors.
The second author used \textsf{GRAPE} \cite{GRAPE}, making
particular use of  its included
function \texttt{SmallestImageSet} \cite{L04}, written by Steve Linton,
which, given a permutation group $A\le S_n$ and a
subset $S$ of $\{1,\ldots,n\}$, determines the lex-min set in the 
$A$-orbit of $S$ (without explicitly constructing this orbit). 

To find the automorphism groups of graphs, we use \emph{nauty} \cite{MP} in the 
current work. Isomorphism and automorphism problems for incidence 
structures encountered (such as sets of lines) are also conveniently 
handled by \emph{nauty} after standard reductions. Furthermore, the 
\texttt{SmallestImageSet} function in \textsf{GRAPE} is very useful
for testing isomorphism when isomorphism corresponds to
two sets being in the same group orbit. 

Note that, due to the transitivity of $\Gr$ on the points,
our 36 lex-min $U$-orbit representatives form a set of
representatives for the $\Gr$-orbits of all the bundles through all the points.
We consider two bundles to be \emph{isomorphic} if they are in the same 
$\Gr$-orbit. 
In the subsequent computations, it will be important to be able to
distinguish between the 36 isomorphism classes of bundles 
with some invariant that can be
computed reasonably quickly. Brouwer \cite{B4} describes
the following invariant, which partitions the bundles into
11 classes.

Count the number of 4-sets $F$ of lines in the bundle through $p$ that 
satisfy the following property: there are four additional candidate lines 
through $p$ that intersect each line in $F$ in some point other than $p$. 
It follows that these four additional
candidate lines must pairwise intersect only in $p$, for otherwise,
we would have a triangle in the generalised quadrangle whose
point graph is induced on the neighbors in $\G$ of $p$. 
Hence, with $p$ removed, the eight lines consisting of
a 4-subset $F$ of the bundle together with four candidate
lines meeting each element of $F$ in a point form a $4\times 4$ ``grid".
The value
of this invariant is between 0 and 63 for the 36 isomorphism
classes. (This subdivision was used also by Reichard.)

We shall now consider a new invariant that is able to distinguish
between all 36 isomorphism classes of bundles. 
Let $p$ be a point and $P$ be the 
set of points that are not adjacent in $\G$ to $p$. Given a bundle through $p$
consider the set of candidate lines $L$ that do not contain
$p$ but intersect two points of some line in the bundle through
$p$. The number of such candidate lines is obviously 
$28\cdot 6 \cdot 9 = 1512$. Each candidate line in $L$ covers three pairs
of points in $P$. Moreover, each pair of points in $P$ is covered
by 0 to 10 candidate lines in $L$. 

The values $t_1,\ldots,t_{10}$, where $t_i$ is the number of pairs of points 
in $P$ covered exactly $i$ times are now invariants. 
Clearly $\sum_{i=1}^{10} it_i = 1512\cdot 3 = 4536$. 
The tuple $(t_1,t_2,\ldots ,t_{10})$ turns out to be an almost complete 
invariant --- actually, it will suffice to use only a part of the tuple,
for example, $(t_1,t_2)$ --- leaving only two pairs of isomorphism classes open.

To make the invariant complete one may further have a look at
the candidate lines $L'$ that intersect exactly two of the lines in the bundle in one
point each. 
Let $L'' \subseteq L'$ be the subset of lines that
do not intersect any of the sets in 
$\{C\cap P : C\in L\}$  in more than one point.
Finally, find a line in the bundle that intersects the fewest number $s$ of
lines in $L''$, and use $s$ in the invariant.

The following invariants of the 36 isomorphism classes are listed in 
Table~\ref{tab:bundle}: the old invariant (Old), the values
$s$, $t_1$,\ldots, $t_{10}$ for the new invariant, and
the order $|\mathrm{Stab}|$ of the stabilizer in $\Gr$ of 
a representative bundle.

\begin{table}
\caption{Isomorphism classes of bundles}
\label{tab:bundle}
\begin{center}
\begin{tabular}{rrrrrrrrrrrrrr}\hline
Nr&Old &$s$&$t_1$&$t_2$&$t_3$&$t_4$&$t_5$&$t_6$&$t_7$&$t_8$&$t_9$&$t_{10}$& $|\mathrm{Stab}|$\\\hline
1 &0   &64& 1730&762&288&97&6&0&0&0&0&0 &6\\
2 &0   &80& 1603&840&336&35&21&0&0&0&0&0&7\\
3 &0   &80& 1568&910&266&70&14&0&0&0&0&0&14\\
4 &0   &84& 1568&910&266&70&14&0&0&0&0&0&14\\
5 &0   &68& 1673&875&231&70&28&0&0&0&0&0&14\\
6 &0   &76& 1736&770&301&70&7&7&0&0&0&0 &14\\
7 &0   &70& 1686&981&117&93&18&9&3&0&0&0&18\\
8 &0   &78& 1626&837&360&15&0&9&6&0&0&0 &18\\
9 &0   &54& 1800&810&207&108&0&0&9&0&0&0&18\\
10&0   &64& 1640&873&270&60&18&0&0&0&0&1&18\\
11&0   &82& 1589&903&231&112&0&0&0&0&0&0&21\\
12&0   &84& 1589&903&231&112&0&0&0&0&0&0&21\\
13&1   &66& 1728&852&208&102&0&12&0&0&0&0&16\\
14&1   &50& 1624&994&184&32&32&14&0&0&0&0&16\\
15&1   &73& 1792&752&272&92&0&8&0&1&0&0  &16\\
16&1   &52& 1824&739&264&92&8&4&0&0&0&1  &16\\
17&1   &73& 1824&840&216&30&24&24&0&0&0&0&48\\
18&1   &27& 1896&780&168&120&0&12&0&3&0&0&48\\
19&2   &70& 1708&832&264&57&26&0&2&0&0&0&2\\
20&3   &73& 1648&855&285&51&21&0&2&0&0&0 &3\\
21&3   &59& 1700&870&240&64&12&6&0&3&0&0 &24\\
22&3   &50& 1764&831&240&62&12&12&0&0&0&1&24\\
23&4   &60& 1773&726&315&72&12&3&0&0&0&0&6\\
24&6   &60& 1717&789&309&47&18&6&0&0&0&0&6\\
25&7   &82& 1520&1008&232&52&8&4&0&4&0&0 &16\\
26&7   &68& 1664&876&272&40&16&8&0&2&0&0 &32\\
27&7   &60& 1664&994&144&80&16&0&0&4&0&2 &32\\
28&7   &89& 1368&1080&264&36&0&12&0&0&0&0&48\\
29&9   &0 & 1584&1242&0&36&0&54&0&0&0&0&432\\
30&9   &38& 1968&837&0&216&0&0&0&0&0&3 &432\\
31&15  &64& 1728&808&280&60&8&12&0&0&0&0&16\\
32&15  &48& 1704&792&312&60&0&12&0&0&0&0&48\\
33&15  &78& 1664&908&224&76&0&8&0&4&0&0 &64\\
34&15  &72& 1792&948&96&112&0&12&0&0&0&4&192\\
35&31  &64& 1824&792&192&120&0&12&0&0&0&0&192\\
36&63  &0 & 2016&756&0&252&0&0&0&0&0&0&12096\\\hline
\end{tabular}
\end{center}
\end{table}

Table~\ref{tab:bundle} shows that six of the 36 representative bundles have 
$t_{10}>0$. Hence, for each of these six bundles, 
there is a pair of points which does
not occur in any feasible candidate line (as this pair is 
only covered by candidate lines which cannot be lines in 
a McLaughlin geometry containing that bundle). 
The fact that these six cases
cannot lead to a McLaughlin geometry was discovered
already by Reichard 
(as noted in \cite[p.~252]{S}), but the specific 
details were never published \cite{R2}.

We remark that, 
given a bundle, a $113\times 1540$ part of the
$275\times 1540$ point-line incidence matrix of a McLaughlin geometry 
${\cal G}$ is determined.
We can see this making use of the fact that every point in 
$\G$ not on a given candidate line $C\in {\cal C}$ must be adjacent to exactly two
points in $C$. This implies that if $C_1$ and $C_2$ are distinct candidate lines in 
our bundle with common point $p$, 
and $q\not=p$ is a point in $C_1$, then
$q$ is adjacent to just one point other than
$p$ in $C_2$, and this yields a total of $1512=\binom{28}{2}4$
edges to be covered once each by the lines of 
${\cal G}$ not in our bundle.  
Consequently, the points of the 
bundle induce a partial linear space having 
28 lines of size~5 and 1512 lines of size~2. 

\subsection{Classifying pairs of bundles}
\label{sect:pairs}

Now, for each of the 30 lex-min representative bundles $B$ through $p_1=1$ 
that have $t_{10}=0$,
we will choose a new point $p_2$ and determine all bundles through $p_2$ 
that are compatible with $B$, up to the action of the stabilizer of $p_2$ 
in the $\Gr$-stabilizer of $B$. 
Two bundles are said to be \emph{compatible}
if no line of the first intersects any line of the second in exactly two points.

The order in which the points are completed, and in particular
the choice of $p_2$ for each chosen bundle through $p_1$, are crucial.
We shall now have a
closer look at the cases when $p_1$ and $p_2$ are adjacent and 
nonadjacent.

When $p_1$ and $p_2$ are adjacent, their completion has a
line $\ell = \{p_1,p_2,a,b,c\}$. Given the bundle through
$p_1$, out of the 280 candidate lines through $p_2$,
the infeasible lines through $p_2$ are exactly those lines 
that contain $p_1$ and one other point of $\ell$. This gives a 
total of $280-(1+4\cdot 9) = 243$ feasible lines.

One may also see that the bundle through $p_1$ reduces the number 
of possible bundles through $p_2$ from 17729280 to 1772928, that is, 
by a factor of exactly 10.
This follows from the fact that given $p_1$, out of the ten 
possibilities for the line through $p_1$ and $p_2$, one has been 
fixed. Since the $\Gr$-stabilizer of $p_2$ is transitive on the 
set of $5$-cliques containing that point, 
and the bundle through $p_1$ does 
not restrict any candidate lines through $p_2$ that intersect $\ell$ 
only in that point, the reduction factor is indeed 10.

If $p_1$ and $p_2$ are nonadjacent, then the vertices in $\G$ 
adjacent to both points induce the Gewirtz graph, which is the unique
(56,10,0,2)-srg (see \cite{BH}). Every candidate line through $p_1$ and $p_2$
intersects the vertex set of that graph in exactly 2 vertices,
and a bundle partitions it into 28 pairs. Consequently, given
a bundle through $p_1$, 
there are $280-28 = 252$ feasible lines for a bundle through
$p_2$. (Recall that a 3-clique is in a unique candidate line by
Lemma~\ref{lem:cl}.)

Although the number of feasible lines through
$p_2$ is smaller when $p_1$ and $p_2$ are adjacent than
when they are nonadjacent, it turns out that the number of compatible
bundles is smaller in practice when they are nonadjacent. We do not
have a formal argument for this, but the following probabilistic
approach gives a result surprisingly close to experimental results.

The bundles through $p_1$ and $p_2$ each cover 28 edges out of the
280 edges in the Gewirtz graph. The probability that two sets of
28 randomly chosen edges are distinct is 

\[\frac{\binom{252}{28}}{\binom{280}{28}} 
\approx 0.04454.
\]

This value is smaller than $1/10 = 0.1$. Although this reasoning
cannot be directly applied to the case of more than two completed
points, practical experiments show that independent sets are indeed
a proper choice when extending larger set of points. 

When choosing the second point $p_2$ to complete amongst those vertices 
(other than $p_1$) that
are not adjacent to $p_1$, the possible choices correspond 
to the orbits of these 
points under the stabilizer in $\Gr$ of the given bundle $B$
through $p_1$. Isomorph rejection will be carried out on the 
compatible bundles through the chosen $p_2$, 
using the action of the stabilizer $K$ 
of $p_2$ in the $\Gr$-stabilizer of $B$. 

Let $|\mathrm{Stab}|$ be the order of the stabilizer in $\Gr$ of the 
bundle $B$ through $p_1 = 1$, and let $l$ be the length of that group's
point-orbit containing $p_2$. Then the group $K$ above 
will have order $|\mathrm{Stab}|/l$. Let $N$ be the number of
bundles through $p_2$ that are compatible with the first bundle $B$.
It turns out that most sets of two bundles have only trivial 
symmetries, so the number of solutions up to symmetry is approximately
$N_1 := Nl/|\mathrm{Stab}|$. The point $p_2$ is chosen so that $N_1$
is minimized. The exact number of solutions up to symmetry is
denoted by $N_2$. The Orbit--Stabilizer theorem is used also here
to validate the results. For each of the 30 lex-min isomorphism class 
representative bundles through $p_1$
which were not rejected in the first step,
the compatible bundles through the chosen
$p_2$ were classified and their lex-min $K$-orbit representatives 
were computed by both authors independently, 
using different software and computers,
the second author using \textsf{GRAPE}
and \texttt{SmallestImageSet} on a Linux desktop PC. 
The results were compared and found to be exactly the same. 

In Table \ref{tab:point2} we now list information about each 
case and chosen point $p_2$, ignoring the six cases that were rejected 
in the first step. The notation given above is used for the 
columns of the table. The last two columns, $N_3$ and $H$, 
will be discussed in the next subsection.

\begin{table}
\caption{Choices for second point}
\label{tab:point2}
\begin{center}
\begin{tabular}{rrrrrrrrr}\hline
Nr& $|\mathrm{Stab}|$ & $p_2$ & $N$ & $l$ & $N_1$ & $N_2$ & $N_3$ & $H$\\\hline
1 &6 & 68  & 940320 & 1 & 156720   & 157760 &   72811 & 946560  \\
2 &7 &161  & 750620 & 1 & 107231   & 107246 &   74493 & 750722  \\
3 &14& 16  &1072194 & 1 & 76585    & 77641  &   25505 & 1086974 \\
4 &14& 79  &1055464 & 1 & 75390    & 76474  &   26850 & 1070636 \\
5 &14& 87  &1041676 & 1 & 74405    & 75034  &   27941 & 1050476 \\
6 &14&159  &1078178 & 1 & 77013    & 78153  &   23269 & 1094142 \\
7 &18&136  &1159506 & 1 & 64417    & 65666  &   17619 & 1181988 \\
8 &18& 69  &1032450 & 3 & 172075   & 172266 &   10186 & 3100788 \\
9 &18& 56  &1089000 & 3 & 181500   & 181742 &    7044 & 3271356 \\
10\\                                                     	 
11&21& 27  &1356738 & 1 & 64607    & 64610  &   15233 & 1356810 \\
12&21& 79  &1356738 & 1 & 64607    & 64610  &   14147 & 1356810 \\
13&16&106  & 431840 & 1 & 26990    & 27412  &   25365 & 438592  \\
14&16&172  & 583384 & 1 & 36461    & 37204  &   33452 & 595264  \\
15&16&118  & 504528 & 2 & 63066    & 63066  &   24992 & 1009056 \\
16\\                                                     	 
17&48& 18  & 851600 & 6 & 106450   & 106450 &     846 & 5109600 \\
18&48&126  & 565584 & 3 & 35349    & 35578  &    6749 & 1707744 \\
19&2 &220  & 894832 & 1 & 447416   & 448417 &  288705 & 896834  \\
20&3 &180  & 600438 & 1 & 200146   & 200146 &  176116 & 600438  \\
21&24&190  &1209360 & 1 & 50390    & 50718  &   12620 & 1217232 \\
22\\                                                     	 
23&6 & 68  & 776563 & 3 & 388282   & 388842 &   52122 & 2333052 \\
24&6 &240  & 714792 & 1 & 119132   & 120096 &   90940 & 720576  \\
25&16& 16  &1053336 & 2 & 131667   & 132016 &   20285 & 2112256 \\
26&32&  3  & 402504 & 4 & 50313    & 50313  &   10310 & 1610016 \\
27\\                                                     	 
28&48& 37  & 571676 & 6 & 71459    & 71683  &    1348 & 3440784 \\
29&432& 8  &1048560 &36 & 87380    & 87778  &      67 & 37920096\\
30\\                                                     	 
31&16& 42  & 457198 & 4 & 114300   & 114670 &   20561 & 1834720 \\
32&48& 16  & 485100 & 6 & 60637    & 60746  &    4104 & 2915808 \\
33&64& 73  &1216976 & 2 & 38030    & 38207  &    2705 & 2445248 \\
34\\                                                     	 
35&192& 38 & 402712 & 12& 25169    & 25441  &     250 & 4884672 \\
36&12096&8 & 396552 & 36& 1180     & 1231   &    1160 & 14890176\\\hline
\end{tabular}
\end{center}
\end{table}

\subsection{Completing the search}
\label{sect:backtrack}

We have already seen that six types of bundles need not be
considered at all in the computer search, splitting the overall
search into 30 cases, with one case for each of the 
remaining bundle types, where we shall be searching for 
a McLaughlin geometry containing a bundle of that type.  
It is clear that after a case for a given
bundle type is completed, that bundle type may be \emph{cancelled},
that is, eliminated
from consideration in the cases of the search to follow.  
A bundle of a cancelled type can be recognised via our bundle
isomorphism class invariant. 
The order in which the cases are considered will have an
impact on the overall running time.

The value of $N_2$ in Table~\ref{tab:point2} gives a rough
estimate of the time it takes to complete the cases 
(ignoring possible cancelled cases). It further turns out
that the impact of a cancelled bundle type on the subsequent cases is
roughly inversely proportional to $|\mathrm{Stab}|$. 
Consequently, a reasonable heuristic for ordering the
parts is via the value $H:=N_2|\mathrm{Stab}|$; this value 
is given in the last column of Table~\ref{tab:point2}.
(Earlier approaches of using invariants and splitting up the search
in many parts are, for example, \cite{HOBE,KO0}.)

We define an order for the parts based on $H$, but put case 36 
early since that is a very small case and good for testing.
The order is as follows (including the six cases that need
not be considered at all at this stage): 
\[
\begin{array}{l}
10\rightarrow 16\rightarrow 22\rightarrow 27\rightarrow 
30\rightarrow 34\rightarrow 36\rightarrow 13\rightarrow 
14\rightarrow 20\rightarrow 24\rightarrow 2\rightarrow\\
19\rightarrow 1\rightarrow 15\rightarrow 5\rightarrow 
4\rightarrow 3\rightarrow 6\rightarrow 7\rightarrow 
21\rightarrow 11\rightarrow 12\rightarrow 26\rightarrow\\
18\rightarrow 31\rightarrow 25\rightarrow 23\rightarrow 
33\rightarrow 32\rightarrow 8\rightarrow 9\rightarrow 
28\rightarrow 35\rightarrow 17\rightarrow 29
\end{array}
\]

Having decided the order, the structures enumerated under
$N_2$ in Table~\ref{tab:point2} can immediately be reduced
by removing the bundles whose types will be cancelled
before a given case is handled. 
The numbers of remaining structures, which are seeds for the
final stage, are given in column $N_3$ of the same table.

The idea of considering independent sets of points to complete 
turned out to have further benefits when programming 
the search. For various reasons related to speed and memory requirement, 
we want to predetermine a set of points to consider in the search. Here
we obviously rely on the assumption that it suffices to consider a
relatively small set of points to prove nonexistence. Results by
Reichard \cite[Sect.\ 5.3]{R} give an indication that this is indeed the case.

There are many independent sets in $\G$, so the choice of points is 
not obvious. We use the following approach.
We first find the set $S$ of points other than $p_1$ and $p_2$  
that are neither adjacent to $p_1$
nor to $p_2$. From the parameters of $\G$, we know that $|S| = 105$.
Let $B(p)$ be the number of bundles through $p$
that are compatible with both of the two bundles
through $p_1$ and $p_2$. For a given $n$, we now find an 
independent set $I \subseteq S$, $|I|=n$ that minimizes 
$\sum_{p\in I}B(p)$. We use $n=9$. (Since all pairs of nonadjacent
vertices are in one orbit of the automorphism group of $\G$, 
the number of independent sets $I$ of a given size to consider will not
depend on the partial solution; for $n=9$, it is
18579960.) To get an upper bound on the size of the data
structures needed, one may further require that $\max_{p\in I}B(p)<u$, for 
some constant $u$,
when minimizing $\sum_{p\in I}B(p)$. We here use this additional
condition with $u=10^5$.


The core of the search is basic backtrack search.
Throughout the search, the candidate bundles through each of
the considered points are maintained. On each level of the
search, the point to consider is first determined.
Specifically, a point with the fewest candidates is chosen.
All candidate bundles for that point are then considered in order. 
For each fixed bundle, the candidate sets for the other points 
are reduced and the procedure described here is recursively repeated.

The core task of the backtrack search is to determine whether
two bundles are compatible. At one 
extreme, one could preprocess every pair of candidate bundles
and tabulate one bit for the result (this is essentially about
determining the edges of a graph in which a clique search is 
to be carried out). However, with, say, $300000$ candidates in 
total, the task of determining a matrix with 
$(3\cdot 10^5)^2 = 9\cdot 10^{10}$ entries is time-consuming and,
in particular, a substantial amount of memory is required.

At the other extreme, one may process each pair of bundles
as they appear in the search. This requires that up to
$28\cdot 28 = 784$ pairs of lines be compared, which is a
rather large constant. Actually, this approach was taken when
carrying out a partial check of the results with an independent
program. Between these extremes, there are
several other possibilities, out of which the following
was chosen.

Assume that we compare a bundle through point $p$ with
a bundle through a nonadjacent point $p'$. Now consider the set $P$ 
of points that will lie on lines through $p$ as well
as lines through $p'$.
The size of $P$ equals the parameter $\mu = 56$ for $\G$. 
We have seen earlier that $P$ induces a 
$(56,10,0,2)$-srg in $\G$,  the 
Gewirtz graph, which has $56\cdot 10/2 = 280$ edges.
The 28 lines through $p$ and $p'$ will each contain 
exactly 28 of the 280 edges. We now have the information
needed for the data structure.

Let $I$ be the (independent) set of points to be completed in the
search. For each pair $(B,p)$, where $B$ is a candidate
bundle through $p'$, such that $p',p \in I$ and $p'\not=p$, determine
and save a 280-bit vector with 28 1s indicating which pairs
of points of lines in $B$ are in the common neighborhood of $p'$ and
$p$. When comparing two bundles, it now suffices to compare
two 280-bit vectors. This can be implemented by checking
whether five AND operations on 64-bit integers all
produce 0.  This very fast method of determining compatibility of
bundles was crucial to the success of our search. 

\section{Results}

The backtrack search described in Section~\ref{sect:backtrack}
was carried out on a 256-core cluster with 2.4-GHz
Intel Xeon E5-2665 processors. As the $n=9$ points could not
be completed in any of the cases, we have the following result.

\begin{theorem}
There is no partial geometry with parameters $(s,t,\alpha)=(4,27,2)$, 
that is, there is no McLaughlin geometry.
\end{theorem}

For the 30 cases, we collect the main data in Table~\ref{tab:tree}.
We give the total run time (in core-days), run time per instance
of the case (the number of instances is $N_3$ in 
Table~\ref{tab:point2}), and the maximum number of points completed
in the search. The total hardware run-time was 
approximately 250 core-years, but since in some parts of the 
search virtual cores were used, the times per case given in  Table~\ref{tab:tree}
sum to rather more than this. 

The length of our computer search was at the borderline of what was
doable with the resources available to the authors. Consequently,
a lot of effort was put into both the general approach and the
specific details described in this paper. In retrospect, the authors
feel that more effort should still have been put on one part, namely
on deciding the order in which the bundle types are handled. The order
used, displayed in the beginning of Section~\ref{sect:backtrack}, 
was obtained with a
heuristic, but additional computational experiments might have led to a
different order and faster computations.

\begin{table}
\caption{Search data}
\label{tab:tree}
\begin{center}
\begin{tabular}{rrrrr}\hline
Nr & Time (d) & Time/Instance (min) & Max.\ level\\\hline
1  & 2325     &     46.0      &     6      \\
2  &13958     &    269.8      &     7      \\
3  &  217     &     12.3      &     6      \\
4  &  264     &     14.2      &     6      \\
5  &  301     &     15.5      &     6      \\
6  &  162     &     10.0      &     6      \\
7  &  100     &      8.2      &     6      \\
8  &   45     &      6.4      &     5      \\
9  &   30     &      6.2      &     4      \\
10 &          &               &            \\
11 &   88     &      8.3      &     6      \\
12 &   76     &      7.8      &     5      \\
13 & 6942     &    394.1      &     7      \\
14 & 6619     &    284.9      &     7      \\
15 &  302     &     17.4      &     6      \\
16 &          &               &            \\
17 &    4     &      6.4      &     3      \\
18 &   34     &      7.3      &     5      \\
19 &42822     &    213.6      &     7      \\
20 &78792     &    644.2      &     8      \\
21 &   67     &      7.7      &     6      \\
22 &          &               &            \\
23 &  271     &      7.5      &     5      \\
24 &15988     &    253.2      &     7      \\
25 &   86     &      6.1      &     5      \\
26 &   50     &      7.0      &     6      \\
27 &          &               &            \\
28 &    6     &      6.3      &     4      \\
29 &    1     &      5.9      &     2      \\
30 &          &               &            \\
31 &   96     &      6.7      &     5      \\
32 &   18     &      6.2      &     5      \\
33 &   10     &      5.5      &     5      \\
34 &          &               &            \\
35 &    1     &      5.8      &     3      \\
36 &   66     &     82.3      &     6      \\\hline
\end{tabular}
\end{center}
\end{table}

Correctness is one of the major issues of computer searches \cite{L,Roy}. There
are effective validation methods whenever the solution set for 
an instance of a classification problem is nonempty \cite[Sect.\ 10.2]{KO},
which can be applied to the initial phases of this search as
we have seen. However, for the final phase, which produces no solutions, the
situtation is different. Also notice that the idea of considering
subcases to avoid redundant computations, which was crucial for the
success of the work, means that the following argument \emph{cannot} be used:
if something exists, then it can been found in many ways.

Two measures were taken to minimize the probability of error. 
A computing cluster with error-correcting code (ECC) memory was used
to minimize the probability of hardware error. Moreover, the results of 
a different, basic (and therefore slow) implementation of the final 
backtrack search were compared with the main implementation for a part
of the search tree (comparing very specific details). 

Having answered the existence question of a McLaughlin geometry 
in the negative, one may continue 
and ask for the largest number of 5-cliques (candidate lines) in the McLaughlin graph
such that no two distinct such cliques intersect in two points. Mathon \cite{M97}
found such packings of size 1120, on which one could try to improve.
Finding the exact answer for this problem does not seem to be within
reach.

\section*{Appendix}

In this Appendix we present a \textsf{GAP/GRAPE} logfile specifying the
construction of our standard fixed copy of the McLaughlin graph $\G$,
used in all our computations.
The group $\Gr$, obtained from the \textsf{GAP} library of primitive
groups and given here by explicit generators, is our standard copy 
of the automorphism group of $\G$. The vertex-set of $\G$ is 
$\{1,\ldots,275\}$, and the edge-set is the $\Gr$-orbit of $\{1,2\}$.

Checks are included showing $\Gr$ to have the right order,
$\G$ to be a strongly regular graph with the correct parameters,
$\Gr=\textrm{Aut}(\G)$, and that the maximal cliques   
of $\G$ (maximal with respect to containment)  are 
precisely the cliques of size 5, of which there is a single $\Gr$-orbit, 
which has size 15400.
 
\bigskip

\begin{footnotesize}

\noindent\verb+gap> LoadPackage("grape");+

\noindent\verb+-----------------------------------------------------------------------------+

\noindent\verb+Loading  GRAPE 4.6.1 (GRaph Algorithms using PErmutation groups)+

\noindent\verb+by Leonard H. Soicher (+\verb+http://www.maths.qmul.ac.uk/~leonard/+\verb+).+

\noindent\verb+Homepage: +\verb+http://www.maths.qmul.ac.uk/~leonard/grape/+

\noindent\verb+-----------------------------------------------------------------------------+

\noindent\verb+true+
\vspace{-10pt}
\begin{verbatim}
gap> G:=Group( [ 
>   (  1,  2)(  3,  4)(  5,  7)(  6,  8)(  9, 12)( 10, 13)( 11, 15)( 14, 19)
>     ( 16, 21)( 17, 22)( 18, 24)( 20, 27)( 23, 31)( 25, 33)( 26, 34)( 28, 37)
>     ( 29, 38)( 30, 40)( 32, 43)( 35, 47)( 36, 48)( 39, 51)( 41, 53)( 42, 54)
>     ( 44, 57)( 45, 58)( 46, 60)( 49, 63)( 50, 64)( 52, 67)( 55, 71)( 56, 72)
>     ( 59, 76)( 61, 78)( 62, 79)( 66, 82)( 68, 85)( 69, 86)( 70, 88)( 73, 92)
>     ( 75, 94)( 77, 97)( 80,101)( 81,102)( 83,105)( 84,106)( 87,110)( 89,112)
>     ( 90,113)( 91,115)( 93,118)( 95,120)( 96,121)( 98,123)( 99,109)(100,125)
>     (103,129)(104,130)(107,133)(108,134)(111,137)(114,140)(116,142)(117,128)
>     (119,144)(122,147)(124,149)(126,139)(131,153)(132,154)(135,158)(136,159)
>     (138,161)(141,164)(143,167)(145,170)(146,172)(148,175)(150,177)(151,178)
>     (152,179)(157,183)(160,187)(162,189)(163,190)(165,168)(166,193)(169,195)
>     (171,196)(174,197)(176,192)(180,204)(181,201)(182,206)(188,209)(191,213)
>     (194,216)(198,221)(199,222)(200,208)(202,225)(203,227)(205,230)(207,232)
>     (210,218)(211,233)(212,235)(215,236)(217,238)(219,240)(220,241)(223,244)
>     (224,245)(226,247)(228,249)(229,250)(231,252)(234,255)(237,257)(239,242)
>     (243,259)(246,261)(248,264)(253,267)(254,268)(256,269)(258,266)(260,270)
>     (262,271)(263,265)(272,274)(273,275), 
>   (  1,  3,  5)(  4,  6,  9)(  7, 10, 14)(  8, 11, 16)( 12, 17, 23)
>     ( 13, 18, 25)( 15, 20, 28)( 19, 26, 35)( 21, 29, 39)( 22, 30, 41)
>     ( 24, 32, 44)( 27, 36, 49)( 31, 42, 55)( 33, 45, 59)( 34, 46, 61)
>     ( 38, 50, 65)( 40, 52, 68)( 43, 56, 73)( 48, 62, 80)( 51, 66, 83)
>     ( 53, 69, 87)( 54, 70, 89)( 57, 74, 93)( 58, 75, 95)( 60, 77, 98)
>     ( 63, 81,103)( 64, 76, 96)( 67, 84,107)( 71, 90,114)( 72, 91,116)
>     ( 78, 99,124)( 79,100,126)( 82,104,131)( 85,108,112)( 86,109,135)
>     ( 88,111,138)( 92,117,123)( 94,119,125)( 97,122,148)(101,127,150)
>     (102,128,151)(106,132,155)(110,136,160)(113,139,162)(115,141,165)
>     (118,143,168)(120,145,171)(121,146,173)(129,144,169)(130,152,180)
>     (133,156,182)(134,157,184)(140,163,191)(142,166,172)(147,174,198)
>     (149,176,200)(153,181,205)(158,185,207)(159,186,170)(161,188,210)
>     (164,192,214)(167,194,217)(175,199,223)(177,201,224)(178,202,226)
>     (179,203,228)(183,193,215)(187,208,209)(189,211,234)(190,212,221)
>     (195,218,239)(196,219,233)(197,220,242)(204,229,251)(206,231,253)
>     (216,237,255)(222,243,260)(225,246,262)(227,248,257)(230,238,258)
>     (232,254,269)(235,256,267)(240,249,265)(247,263,261)(250,264,268)
>     (252,266,272)(270,271,273) ]);;
gap> gamma:=EdgeOrbitsGraph(G,[1,2]);;
gap> # Now  gamma  is our standard copy of the McLaughlin graph, 
gap> # with  gamma.group  set to  G.  We do some checks:
gap> Size(G);
1796256000
gap> GlobalParameters(gamma);
[ [ 0, 0, 112 ], [ 1, 30, 81 ], [ 56, 56, 0 ] ]
gap> G=AutomorphismGroup(gamma); 
true
gap> G=gamma.group;
true
gap> # We now determine the maximal cliques of  gamma,  and that these 
gap> # are precisely the cliques of size 5, of which there is a single 
gap> # G-orbit, which has size 15400. 
gap> K:=CompleteSubgraphs(gamma,-1,2); 
[ [ 1, 2, 17, 45, 193 ] ]
gap> # Now  K  is a set of G-orbit representatives of the maximal 
gap> # cliques of  gamma. 
gap> cliques:=Set(Orbit(G,K[1],OnSets));; 
gap> Length(cliques);
15400
\end{verbatim}
\end{footnotesize}

\section*{Acknowledgements}

The first author was supported in part by the Academy of Finland,
Project No. 289002. We thank Frank De~Clerck 
for making the second author aware, about 25 years ago, of the interesting 
and challenging problem of the existence of a McLaughlin geometry. 
Andries Brouwer's web pages with extensive data on many interesting
graphs, especially strongly regular 
graphs, were invaluable in the development of the algorithms and
heuristics.

\end{document}